\documentclass{amsart}

\usepackage{amssymb}
\usepackage{amsfonts}
\usepackage{amsmath}
\usepackage{amsthm}
\usepackage{graphicx}
\usepackage{mathrsfs}
\usepackage{dsfont}
\usepackage{amscd}
\usepackage{multirow}
\usepackage{palatino}
\usepackage{mathpazo}
\usepackage[all]{xy}
\usepackage[scaled]{helvet} 
\usepackage{courier} 
\normalfont
\usepackage[T1]{fontenc}
\usepackage{calligra}
\usepackage{verbatim}
\usepackage[usenames,dvipsnames]{color}
\usepackage[colorlinks=true,linkcolor=Blue,citecolor=Violet]{hyperref}
\usepackage{cite}
\usepackage{enumitem}

\makeatletter
\@namedef{subjclassname@2010}{%
  \textup{2010} Mathematics Subject Classification}
\makeatother

\theoremstyle{plain}
\newtheorem{theorem}{Theorem}[section]

\newtheorem{lemma}[theorem]{Lemma}
\newtheorem{proposition}[theorem]{Proposition}

\newtheorem{theorem-definition}[theorem]{Theorem-Definition}

\theoremstyle{definition}
\newtheorem{definition}[theorem]{Definition}

\theoremstyle{remark}

\newtheorem{example}[theorem]{Example}

\newtheorem{remark}[theorem]{Remark}

\numberwithin{equation}{section}
 
\newcommand{\la}{{\lambda}}
\newcommand{\cC}{{\mathcal{C}}}

\newcommand{\N}{{\mathds{N}}}

\newcommand{\Q}{{\mathds{Q}}}
\newcommand{\R}{{\mathds{R}}}
\newcommand{\C}{{\mathds{C}}}

\newcommand{\A}{{\mathcal{A}}}
\newcommand{\B}{{\mathcal{B}}}
\newcommand{\M}{{\mathcal{M}}}

\newcommand{\op}{\mathrm{op}}
\newcommand{\diag}{\mathrm{diag}}

\renewcommand{\geq}{\geqslant}
\renewcommand{\leq}{\leqslant}

\allowdisplaybreaks[1]

\makeatletter
\newcommand{\vast}{\bBigg@{4}}
\newcommand{\Vast}{\bBigg@{5}}
\makeatother

\begin{document}

\title[Frobenius--Rieffel norms]{  Frobenius--Rieffel norms on finite-dimensional C*-algebras}
\author{Konrad Aguilar}
\address{Department of Mathematics and Statistics, Pomona College, 610 N. College Ave., Claremont, CA 91711} 
\email{konrad.aguilar@pomona.edu}
\urladdr{\url{https://aguilar.sites.pomona.edu/}}
\thanks{The first author gratefully acknowledges the financial support  from the Independent Research Fund Denmark  through the project `Classical and Quantum Distances' (grant no.~9040-00107B)}

\author[S.R.~Garcia]{Stephan Ramon Garcia}
\address{Department of Mathematics and Statistics, Pomona College, 610 N. College Ave., Claremont, CA 91711} 
\email{stephan.garcia@pomona.edu}
\urladdr{\url{http://pages.pomona.edu/~sg064747}}
\thanks{The second author is partially supported by NSF grants DMS-1800123 and DMS-2054002}

\author{Elena Kim}
\address{Department of Mathematics, Massachusetts Institute of Technology, 77 Massachusetts Ave, Cambridge, MA 02139}
\email{elenakim@mit.edu}
\urladdr{}
\thanks{The third author is supported by the National Science Foundation Graduate Research Fellowship under Grant No. 1745302}

\date{\today}

\begin{abstract}
In 2014,  Rieffel introduced norms on certain unital C*-algebras built from conditional expectations onto unital C*-subalgebras. We   begin by showing that these norms generalize the Frobenius norm, and we provide explicit formulas for certain conditional expectations onto unital C*-subalgebras  of finite-dimensional C*-algebras. This allows us compare  these norms to the  operator norm by finding explicit equivalence constants. In particular, we find equivalence constants for the standard finite-dimensional C*-subalgebras of the Effros--Shen algebras that vary continuously with respect to their given irrational parameters.
\end{abstract}
\maketitle

\section{Introduction}
A main goal of noncommutative metric geometry is to establish the convergence  of spaces arising in the physics or operator-algebra literature \cite{Rieffel00,Rieffel01, Latremoliere13c, Junge18, Kaad20, Latremoliere20}. To accomplish this, one must equip operator algebras with compact quantum metrics, which were introduced by Rieffel \cite{Rieffel98a,Rieffel99} and motivated by work of Connes \cite{Connes89, Connes}. Then,  convergence of   compact quantum metric spaces is proven with quantum analogues of the Gromov--Hausdorff distance \cite{Rieffel00,Latremoliere13, Latremoliere13b, Kerr09, Li06, Wu}.

 In \cite{Aguilar-Latremoliere15}, the first author and Latr\'emoli\`ere  recently exhibited the convergence of quantum metric spaces built from   approximately finite-dimensional C*-algebras (AF algebras) and, in particular, convergence of the Effros--Shen algebras \cite{Effros80b} with respect to their irrational parameters. Quantum metric spaces are obtained by endowing
unital C*-algebras with a type of seminorm whose properties are inspired
by the Lipschitz seminorm. A property, which is not needed in \cite{Aguilar-Latremoliere15}, but
appears desirable in other context \cite{Rieffel10}, is   called the  strongly Leibniz property. A seminorm $s$ satisfies the {\em strongly Leibniz property} on an operator algebra $\A$ if  
 \[
 s(A^{-1})\leq s(A) \cdot \|A^{-1}\|_\mathrm{op}^{2}
 \]
  for all invertible $A \in \A$, in which $\|\cdot\|_\mathrm{op}$ is the operator norm. 
 This can be seen as a noncommutative analogue of the   quotient rule for derivatives. Although the authors of \cite{Aguilar-Latremoliere15} were able to prove their results without this property, Rieffel’s work on module
convergence over the sphere \cite{Rieffel10} uses the strong Leibniz property, and it
can be expected to play a role in the study of module convergence in
general. 
 
 Let $T$ be a topological space and let $t \in T$. The reason that the seminorms in \cite{Aguilar-Latremoliere15} do not likely satisfy the strongly Leibniz rule is  because the  seminorms are of the form
 \[
   A \longmapsto \|A-P_t(A)\|_\mathrm{op},
 \]
 where $A$ is an element of the C*-algebra $\A$, $\B$ is a C*-subalgebra of $\A$, and  $P_t:\A\rightarrow \B$ is a certain surjective linear map called a  faithful conditional expectation. But conditional expectations are rarely multiplicative (otherwise, the strongly Leibniz property of this seminorm would come for free). Rather than replace $P_t$, which provides crucial estimates, Rieffel provided another option in \cite[Section 5]{Rieffel12} following his previous work in \cite{Rieffel74}: replace the operator norm  with one induced by $P_t$ and the subalgebra $\B$.  For  $A \in \A$, the {\em  Frobenius--Rieffel} norm is
 \[
 \|A\|_{P_t}=\sqrt{\|P_t(A^*A)\|_\mathrm{op}},
 \]
 where $A^*$ is the adjoint of $A$. If we define 
 \[
s_{P_t}: A \longmapsto \|A-P_t(A)\|_{P_t},
 \]
 then $s_{P_t}$ is a seminorm that 
 is strongly Leibniz   \cite[Theorem 5.5]{Rieffel12}. 
 
 However, this replacement comes at a cost.   Following \cite{Aguilar-Latremoliere15},   we want the   family of maps $(s_{P_t})_{t \in T}$ to vary continuously (pointwise) on a particular subset of $\A$ with respect to $\|\cdot\|_\op$.  Thankfully,  in the setting of \cite{Aguilar-Latremoliere15}, one need only verify this continuity   when $\A$ is finite dimensional. In this case, $\|\cdot\|_{P_t}$ and $\|\cdot\|_\op$ are equivalent on $\A$, meaning there exist constants $\kappa_t^+, \kappa_t^->0$ such that
 \[
 \kappa_t^+ \|\cdot \|_{\op} \leq \|\cdot\|_{P_t} \leq \kappa_t^- \|\cdot\|_\op
 \] for each $t \in T$. Therefore, we can replace $s_{P_t}$ with   
 \[
 A \longmapsto \frac{1}{\kappa_t^+} s_{P_t}(A),
 \] 
 which is strongly Leibniz. However, the constants  $\kappa_t^+,\kappa_t^-$ need not change continuously with respect to $t$.  Therefore, our aim in this paper is to find explicit equivalence constants  for the  operator norm and   Frobenius--Rieffel norms on finite-dimensional C*-algebras, so that we may prove the continuity of the constants $\kappa^\pm_t$ with respect to $t \in T$. In fact, one of our main results (Theorem \ref{e:c-frac}) shows that there exist explicit equivalence constants for the finite-dimensional C*-algebras that form the Effros--Shen algebras which vary continuously with respect to the irrational parameters that determine these algebras.
 
 After some background on C*-algebras and the construction of the Frobenius--Rieffel norms, we provide some basic facts in the next section. Then, we  find equivalence constants when $\A$ is     the space of complex $n\times n$-matrices. This provides a framework for the general case of finite-dimensional C*-algebras, which we tackle next. Our main method is to represent the conditional expectations as means of unitary conjugates for some standard subalgebras, and then extend these results to all unital C*-subalgebras by showing that although the Frobenius--Rieffel norms are not unitarily invariant, their equivalence constants are.
 
 \section{Preliminaries} The facts we state about C*-algebras in this section can be found in  standard   texts such as \cite{Davidson, Murphy90, Pedersen79}.
A {\em C*-algebra} $(\A, \|\cdot\|)$ is a Banach algebra over $\C$  equipped with  a conjugate-linear anti-multiplicative involution $*: \A \rightarrow \A$ called the {\em adjoint} satisfying  the C*-identity (i.e., $
\|A^*A\|=\|A\|^2$ for all $A \in \A$). 
We say that $\A$ is {\em  unital} if it has a multiplicative identity. If two C*-algebras $\A, \B$ are *-isomorphic, then we denote this by  $\A\cong \B$.  Let $n \in \N=\{1, 2, 3, \ldots\}$. We denote the space of complex $n \times n$    matrices by $\M_n$ and its C*-norm  by  $\|\cdot\|_{\mathrm{op}}$, the operator norm induced by the 2-norm on $\C^n$. We denote the $n\times n$ identity matrix by $I_n$. For $A \in \M_n$, we let $A_{i,j} \in \C$   denote the $(i,j)$-entry of $A$ for all $i,j \in [N]$, where $[N]=\{1, 2, \ldots, N\}$.
 
\begin{example}\label{e:fd} Let $N \in \N$ and let   $d_1, d_2,\ldots, d_N \in \N$. The space
\[
\bigoplus_{k=1}^N \M_{d_k}
\]
is a unital C*-algebra with coordinate-wise operations; the norm is the maximum of the operator norms in each coordinate. If we set $n=d_1+d_2+\cdots+d_N$, then $I_n=\bigoplus_{k=1}^N I_{d_k}$ is the unit,  which we frequently denote by $I$. Every finite-dimensional C*-algebra is of this form up to *-isomorphism   \cite[Theorem III.1.1]{Davidson}.

We denote $A \in \bigoplus_{k=1}^N \M_{d_k}$ by $A=(A^{(1)}, A^{(2)}, \ldots, A^{(N)})$, so that $A^{(k)} \in \M_{d_k}$ for each $k \in [N]$ and  $A^{(k)}_{i,j} \in \C$ is the  $(i,j)$-entry of $A^{(k)}$ for all $i,j \in [d_k]$.
\end{example}

The following maps are needed for the construction of Frobenius--Rieffel norms.

\begin{definition}\label{d:ce}
Let $\A$ be a unital C*-algebra and let $\B \subseteq \A$ be a unital C*-subalgebra.  A linear function $P: \A \rightarrow \B$ is a {\em conditional expectation} if  
\begin{enumerate}[leftmargin=*]
\item $\forall B \in \B$, $P(B)=B$, and 
\item $\forall A \in \A$, $\|P(A)\|\leq \|A\|$.
\end{enumerate} 
We say that $P$ is {\em faithful} if $P(A^*A)=0\iff A=0$.
\end{definition} 
We can now define   norms induced by faithful conditional expectations.

\begin{theorem}[\!{\cite{Rieffel74}}{\cite[Section 5]{Rieffel12}}]\label{td:gen-Frob-norm}
Let $\A$ be a unital C*-algebra and $\B \subseteq \A$ be a unital C*-subalgebra. Let $P: \A \rightarrow \B$ be a faithful conditional expectation. 
For all $A \in \A$, set
\[
\|A\|_{P,\B}=\sqrt{\|P(A^*A)\|}.
\]
This defines a norm on $\A$ called the {\em  Frobenius--Rieffel norm associated to $\B$ and $P$}.
\end{theorem}
  The terminology for these norms is due to two facts: Rieffel introduced these norms \cite[Section 5]{Rieffel12} using his work that introduced spaces called Hilbert C*-modules \cite{Rieffel74}  and we show in Theorem \ref{A_1^n}    that one can recover the Frobenius norm using a particular C*-subalgebra.
 
One of the main results that makes our work in this paper possible is the fact that we can express our conditional expectations as orthogonal projections.  The key property that allows   this is the preservation of faithful tracial states. A {\em state} on a C*-algebra $\A$ is a positive linear functional $\varphi: \A \rightarrow \C$ of norm $1$. We say that $\varphi$ is {\em faithful} if $\varphi(A^*A)=0 \iff A=0$ and  {\em tracial} if   $\varphi(AB)=\varphi(BA)$ for all $A,B \in \A$.   If $\B$ is a unital C*-subalgebra and $P: \A \rightarrow \B$ is a conditional expectation onto $\B$, then we say that $P$ is {\em $\varphi$-preserving} if $\varphi \circ P=\varphi$.

\begin{example}\label{e:trace}
Let $N \in \N$ and $d_1, d_2\ldots, d_N \in \N$. Let $\mathbf{v}=(v_{1},v_2, \ldots, v_{N}) \in (0,1)^N$ such that $\sum_{k=1}^N v_{k}=1$. For every $A=(A^{(1)}, \ldots, A^{(N)}) \in 
\bigoplus_{k=1}^N \M_{d_k}$, define
\[
\tau_\mathbf{v}(A)=\sum_{k=1}^N \frac{v_{k}}{d_k} \mathrm{Tr}(A^{(k)}),
\]
where $\mathrm{Tr}$ is the trace of a matrix. 
Then $\tau_\mathbf{v}$ is a faithful tracial state on $ \bigoplus_{k=1}^N \M_{d_k} $. In fact, all faithful tracial states on $ \bigoplus_{k=1}^N \M_{n_k} $ are of this form  \cite[Example IV.5.4]{Davidson}. For $\M_n$, we have   $\mathbf{v}=(1)$. Thus, $\tau_\mathbf{v}=\frac{1}{n} \mathrm{Tr}$, and we simply denote $\tau_\mathbf{v}$ by $\tau$ in this case.
\end{example}
A faithful tracial state allows us to define an inner product on $\A$.

\begin{theorem}[\!{\cite[Proposition VIII.5.11]{Conway90}}]\label{t:trace-inner}
Let $\A$ be a unital C*-algebra  and let $\varphi : \A \rightarrow \C$ be a faithful  state. Then
\[
\langle A,B\rangle_\varphi=\varphi(B^*A)
\]  is an inner product on $\A$.
\end{theorem}
  The following fact is well known.

\begin{theorem}[\!{\cite[Expression (4.1)]{Aguilar-Latremoliere15}}]\label{t:proj}
Let $\A$ be a unital C*-algebra, let $\B \subseteq \A$ be a unital C*-subalgebra, and let $\varphi: \A \rightarrow \C$ be a faithful tracial state. 
If $\B$ is finite dimensional, then there exists a unique $\varphi$-preserving conditional expectation $P^\varphi_\B: \A \rightarrow \B$ onto $\B$ such that given any basis  $\beta$ of $\B$ which is orthogonal with respect to $\langle \cdot, \cdot\rangle_\varphi$, we have
\[
P^\varphi_\B(A)=\sum_{B \in \beta} \frac{\langle A,B\rangle_\varphi }{\langle B,B \rangle_\varphi}B
\]
for all $A \in \A$.
\end{theorem}
In this case, we denote the associated Frobenius--Rieffel norm on $\A$ by $\|\cdot\|_{\tau,\B}$.
  Now, let $\A=\bigoplus_{k=1}^N \M_{d_k}$, let $\B \subseteq \A$ be a unital C*-subalgebra, and let $\mathbf{v}=(v_1, v_2, \ldots, v_N)\in (0,1)^N$ such that $\sum_{k=1}^Nv_k=1$.
\begin{enumerate}[leftmargin=*]
\item We denote the  conditional expectation of Theorem \ref{t:proj} induced by the faithful tracial state $\tau_\mathbf{v}$ of Example \ref{e:trace} by $P^\mathbf{v}_\B$. We denote the associated  Frobenius--Rieffel norm by $\|\cdot\|_{\mathbf{v},\B}$.

\item If $N=1$, then $\tau=\frac{1}{d_1} \mathrm{Tr}$   is the unique faithful tracial state on $\M_{d_1} $, and we denote the conditional expectation of Theorem \ref{t:proj} induced by $\tau$ by $P_\B$. We denote the associated Frobenius--Rieffel norm by $\|\cdot\|_{\B}$.
\end{enumerate}

\section{Some properties of Frobenius--Rieffel norms}

In this section, we detail the subalgebras of $\bigoplus_{k=1}^N \M_{d_k}$ that we will be working with and the conditional expectations given by  Theorem \ref{t:proj}. We   also explain   why we use ``Frobenius'' in the name of the norms  of Theorem  \ref{td:gen-Frob-norm}. 

A {\em partition} $\la$ of  $n \in \N$, denoted $\la \vdash n,$ is a tuple $\la=(n_1, n_2, \ldots, n_L) \in \N^L$, where $L \in \N$ depends on $n$ and $n=\sum_{i=1}^L n_i$. We need  the following refinement   to describe certain subalgebras of $\M_n$. 
\begin{definition}\label{d:part}
Let $n,L \in \N$. A formal expression $\la = (n_1^{m_1},n_2^{m_2},\ldots, n_{L}^{m_{L}}),$ in which $m_i, n_i \in \N$  for $1 \leq i \leq L,$ and
 \begin{equation}\label{eq:ref-p}n = \sum_{i=1}^{L} m_i n_i    \end{equation}
is a {\em refined partition} of $n$, denoted $\langle \la \vdash n\rangle$.  Write $\mathbf{m}_{\la} = (m_1,m_2, \ldots,m_{L})$ and $\mathbf{n}_{\la} = (n_1,n_2,\ldots,n_{L}),$ so that $n = \mathbf{m}_{\la} \cdot \mathbf{n}_{\la}$. The vectors $\mathbf{m}_{\la}$ and $\mathbf{n}_{\la}$ are the {\em multiplicity vector} and {\em dimension vector} of $\la$, respectively.  We drop the subscript $\la$ unless   needed for clarity. In the formal expression for $\la$, we suppress $m_i$ if $m_i=1$. The number $L=L(\la)$ of summands in \eqref{eq:ref-p} is the {\em length} of $\la$.
\end{definition}
For example, $(2^2,2), (2^3), (2,2,1^2), (3,3),(6)$ are refined partitions of $6$ with, respectively, lengths $2,1,3,2,1$;    multiplicity vectors $(2,1), (3), (1,1,2), (1,1), (1)$; and dimension vectors $(2,2), (2), (2,2,1), (3,3), (6)$.

In what follows, we use Kronecker products and direct sums. For example,     by $(I_2 \otimes \M_2)\oplus (I_1 \otimes \M_1) \subset \M_5$, we mean the subalgebra  
\[
\left\{ \mathrm{diag}(A,A, \mu)   : A \in \M_2, \mu \in \M_1\right\}
\]
of $\M_5$, 
where $\mathrm{diag}(A,A, \mu)$ is the block-diagonal matrix 
\[
\begin{bmatrix}
A & & \\
 & A & \\
  & & \mu
\end{bmatrix},
\]
with $0$s in the entries not occupied by the $A$s and $\mu$.
\begin{definition}\label{n:s-full} Let $n \in \N$ and let  $\la$ be a refined partition of $n$. The {\em subalgebra of $\M_n$ corresponding to $\langle \la \vdash n \rangle$} is 
\begin{equation}\label{eq:s-full}\B^n_\la=
\bigoplus_{i=1}^{L(\la)} \left(I_{m_i} \otimes \M_{n_i}\right).\end{equation}
We   sometimes write $\B_\la$ instead of $\B^n_\la$ when the context is clear.
\end{definition}

\begin{example}\label{e:frob-sub}
For each   $n \in \N$,  we have   $\B^n_{1^n} =I_n \otimes \M_1 \cong \M_1$ and
\[\B^n_{1, 1, \ldots, 1}= \bigoplus_{i=1}^n (I_1 \otimes  \M_1)=\{\mathrm{diag}(\mu_1, \mu_2, \ldots, \mu_n) : \mu_1, \mu_2, \ldots, \mu_n \in \C\} \cong \C^n,\]
where   $n$-copies of $1$ are in the subscript of $\B^n_{1, 1, \ldots, 1}$ and $\cong$ denotes *-isomorphism.
\end{example}
\begin{example}
Observe that 
\begin{align*}
\B^4_{1^2,2}&=(I_2 \otimes \M_1) \oplus (I_1 \otimes \M_2) =\left\{ \diag\left(\mu,\mu,A \right)   : \mu \in \C, A \in \M_2 \right\}\cong \M_1 \oplus \M_2.
\end{align*}
Thus,
 \[ \left\{ \diag\left(\mu,A,\mu \right)   : \mu \in \C, A \in \M_2  \right\} 
\] is a unital C*-subalgebra of $\M_4$ which is not of the form   \eqref{eq:s-full}, but is *-isomorphic to $
\B^4_{1^2,2}$.
\end{example}

The algebra of circulant matrices   provides another example of a unital C*-subalgebra of $\M_n$  that is not of the form   \eqref{eq:s-full}.
\begin{example} 
A matrix of the form 
\[
\begin{bmatrix}
a_1 & a_2 & a_3 & \cdots & a_n\\
a_n & a_1 & a_2 & \cdots & a_{n-1}\\
a_{n-1} & a_n & a_1 & \cdots & a_{n-2}\\
\vdots & \vdots & \ddots & \ddots & \vdots\\
a_2 & a_3 & \cdots & a_n & a_1
\end{bmatrix}
\]
is a {\em circulant matrix } \cite[0.9.6 and 2.2.P10]{Horn}\label{e:circulant}.  The *-algebra of $n \times n$ circulant matrices is a unital commutative  C*-subalgebra of $\M_n$ that is *-isomorphic to $\B^n_{1, 1,\ldots, 1}$. Indeed, they are simultaneously unitarily diagonalizable normal matrices.
\end{example}
  The next definition   serves as a vital intermediate step in finding equivalence constants associated to  all unital C*-subalgebras and faithful tracial states of $\bigoplus_{k=1}^N \M_{d_k}$.  
\begin{definition}\label{d:s-sub}
Consider $\A=\bigoplus_{k=1}^N \M_{d_k}$. For each $k \in [N]$, let \[p_k : \A \rightarrow \M_{d_k} \] be the canonical projection onto the $k$th summand. 
We say that  $\B\subseteq \A$ is a {\em standard unital C*-subalgebra} if it is a unital C*-subalgebra such that for each $k \in [N]$ 
\[
p_k(\B)=\B^{d_k}_{\la_k},
\]
where $\langle \la_k \vdash d_k\rangle$. 
Then $\B$ is a unital C*-subalgebra of 
\[\cC_\B=\bigoplus_{k=1}^N \B^{d_k}_{\la_k},\] which  is a unital C*-subalgebra of $ \A$.
\end{definition}

\begin{example}\label{e:d-sum}
Observe that 
\[
\B=\left\{ \diag(\mu,\nu) \oplus \mu\in \M_2 \oplus \M_1 :  \mu,\nu \in \C \right\}\cong \C^2
\]
is a standard unital C*-subalgebra of $\M_2 \oplus \M_1$  and 
\[
\cC_\B=\left\{ \diag(\mu,\nu) \oplus \eta\in \M_2 \oplus \M_1 :  \mu,\nu,\eta \in \C \right\}=\B^2_{1,1} \oplus \B^1_1\cong \C^3.
\]
We note that the unital C*-subalgebra of $\M_3$ given by
\[
\{ \diag(\mu,\nu,\mu): \mu,\nu \in \C\}
\]
is not standard, but it is *-isomorphic to $\B$. Thus, whether a subalgebra is standard or not depends upon the larger ambient algebra.
\end{example}

\begin{example}\label{e:more-s-alg}
Observe that 
\[
\B= \left\{ \diag(\mu,A,\mu) \oplus A \in \M_4 \oplus \M_2 : \mu\in\C, A \in \M_2\right\}
\]
is not a standard unital C*-subalgebra of $\M_4 \oplus \M_2$ since $p_1(\B)=\{\diag(\mu,A,\mu): \mu\in \C, A \in \M_2\}$ is not of the form   \eqref{eq:s-full}. But it is *-isomorphic to the standard unital C*-subalgebra  \[\mathcal{E}=\{ \diag(A,\mu,\mu) \oplus A \in \M_4 \oplus \M_2 : \mu \in \C, A \in \M_2\}\] of $\M_4 \oplus \M_2$. Note $p_1(\mathcal{E})=\B_{2,1^2}$ and $p_2(\mathcal{E})=\B_2=\M_2$.
\end{example}

Up to unitary equivalence, standard unital C*-subalgebras  comprise all unital C*-subalgebras of $\bigoplus_{k=1}^N \M_{d_k}$.
To be clear, let $\B, \cC \subseteq \A=\bigoplus_{k=1}^N\M_{d_k}$ be two unital C*-subalgebras. We say that $\B$ and  $\cC$ are {\em unitarily equivalent} (with respect to $\A$)  if there exists a unitary $U \in \A$ such that  $B \mapsto UBU^*$ is a bijection from $\B$ onto $\cC$, in which case  we write $\cC=U\B U^*$.   Sometimes the term {\em  spatially isomorphic} is used for unitary equivalence, but spatially isomorphic is also sometimes used in a more general sense.

Unitary equivalence  is stronger than *-isomorphism. For example, the   unital C*-subalgebras  
\[
\B_{1^2}\oplus \B_{1,1} \quad \text{ and } \quad \B_{1,1}\oplus \B_{1^2}
\]
of $\M_2\oplus \M_2$ 
are *-isomorphic but not unitarily equivalent    in $\M_2 \oplus \M_2$ (they are unitarily equivalent in $\M_4$, but we are viewing them as subalgebras of $\M_2 \oplus\M_2$).   We now state the following well-known result.

\begin{theorem}[\!{\cite[Theorem III.1.1, Corollary III.1.2, and Lemma III.2.1]{Davidson}}]\label{t:all-sub}
Every unital C*-subalgebra $\B\subseteq \M_n$ is unitarily equivalent with respect to $\M_n$ to $\B_\la$   for some refined partition $\la=(n_1^{m_1}, n_2^{m_2}, \ldots, n_L^{m_L})$ of $n$, and   
\[  \B_\la \cong \bigoplus_{i=1}^{L} \M_{n_i} .\]
Furthermore, any unital C*-subalgebra of  $ \bigoplus_{k=1}^N\M_{d_k}$ is  unitarily equivalent, with respect to $ \bigoplus_{k=1}^N\M_{d_k}$, to a standard unital C*-subalgebra.
\end{theorem}
For example, the *-algebra of circulant matrices of Example \ref{e:circulant} is unitarily equivalent with respect to $\M_n$ to $\B^n_{1,1, \ldots, 1}$, not just *-isomorphic to it \cite[2.2.P10]{Horn}. Also, the subalgebras $\B$ and $\mathcal{E}$ of Example \ref{e:more-s-alg} are unitarily equivalent, not just *-isomorphic.

We   use Theorem \ref{t:all-sub} to generalize our results to all unital C*-subalgebras once we verify our results for the standard   subalgebras. One of the advantages of working with standard unital C*-subalgebras is that they have  canonical bases which are orthogonal with respect to the inner products induced by faithful tracial states.

 \begin{definition}
 Consider   $  \bigoplus_{k=1}^N \M_{d_k} $.  For each $k \in [N]$ and $ i,j \in [ d_k]$, let  $E^{(k)}_{i,j} \in \bigoplus_{k=1}^N \M_{d_k}$ have a $1$ in the $(i,j)$-entry of the $k$th summand and zeros in all other  entries and all other summands. We call   $E^{(k)}_{i,j}$  a {\em matrix unit}. If $N=1$, then we  suppress the superscript $(k)$.
\end{definition}   
Any standard unital C*-subalgebra  $\B \subseteq \bigoplus_{k=1}^N \M_{d_k}$   has a     standard basis (up to ordering of terms) given by matrix units or sums of distinct matrix units, which we denote by $\beta_\B$. 
\begin{example}\label{e:d-sum-basis}
For the subalgebra $\B\subseteq \M_2\oplus \M_1$ of Example \ref{e:d-sum}, we have  $\beta_\B=\{E^{(1)}_{1,1}+E^{(2)}_{1,1}, E^{(1)}_{2,2}\}$.
\end{example}
\begin{example}\label{e:s-sub-basis}
For $\B^n_{1^n}\subseteq \M_n$, we have  $\beta_{\B_{1^n}}=\left\{\sum_{k=1}^n E_{k,k}\right\}=\{I_n\}.$

For $\B^n_{1,1, \ldots, 1}\subseteq \M_n$, we have 
$
\beta_{\B^n_{1, 1,\ldots, 1}}=\{E_{1,1}, E_{2,2}, \ldots, E_{n,n}\}.
$

For $\B^4_{2^2}\subseteq \M_4$, we have   $\beta_{\B^4_{2^2}}=\{E_{i,j}+E_{i+2,j+2}:i,j \in [2]\}$.
\end{example}
All cases in the    example above can be recovered as follows.
\begin{remark}Let $\B \subseteq \M_n$ be a standard unital C*-subalgebra of Definition \ref{d:s-sub}. Thus, there exists a  refined partition $\la=(n_1^{m_1}, n_2^{m_2}, \ldots, n_L^{m_L})$ of $n$ such that  $\B=\B^n_\la$, and we have that 
\begin{align*}
 \beta_{\B^n_\la}=\bigcup_{k=1}^L \left\{\sum_{t=0}^{m_k-1} E_{i_t(p),j_t(q)} : p,q \in [ n_k],  i_t(p)-p=j_t(q)-q=tn_k+\sum_{r=1}^{k-1} m_rn_r \right\},
\end{align*}
where we regard a sum over an empty set of indices as zero.
\end{remark}

  For these bases, although some of the elements are sums of distinct matrix units, we note that the summands are from distinct blocks. For example,  for $\B^4_{2^2}$, no  element of the form $E_{i,j}+E_{i,m}$ appears in the standard basis. This easily verified fact and more are summarized in the following. 
\begin{theorem}\label{t:can-basis}
Let $\B \subseteq \bigoplus_{k=1}^N \M_{d_k}$ be a standard unital C*-subalgebra.   For each $B \in \beta_\B$, let $\Psi_B=\{(k;i,j): B^{(k)}_{i,j}=1 \}$ (i.e., $ B=\sum_{(k;i,j)\in \Psi_B} E^{(k)}_{i,j}$). 
The following hold:
\begin{enumerate}[leftmargin=*]
\item $\Psi_B \cap \Psi_{B'}=\emptyset$ for every $B,B' \in \beta_\B $ with $ B \neq B'$.
\item If $(k;i,j),(k';i',j') \in \Psi_B$, then $(k;i,j)=(k';i',j')$ if and only if $k=k'$ and ($i=i'$ or $j=j'$).
\item If $\mathbf{v}=(v_1, v_2, \ldots, v_N)\in (0,1)^N$, then $\beta_\B$ is an orthogonal basis of $\B$ with respect to $\langle \cdot, \cdot\rangle_{\tau_\mathbf{v}}$.
\end{enumerate}
\end{theorem}

 We now provide an explicit way of calculating the conditional expectations associated with standard unital C*-subalgebras. This is a complete generalization of   \cite[Proposition 2.8]{Aguilar-Brooker19}.

\begin{theorem}\label{t:proj-formula}
Let $\B \subseteq \A=\bigoplus_{k=1}^N \M_{d_k}$ be a standard unital C*-subalgebra.  For each $B \in \beta_\B$ and $k \in [N]$, let $\Psi^{(k)}_B=\{(i,j): (k;i,j) \in \Psi_B\}$, and   let $|\Psi^{(k)}_B|$ denote  the cardinality of $\Psi^{(k)}_B$. Let $\mathbf{v}\in (0,1)^N$ satisfy $\sum_{k=1}^Nv_k=1$. If $A \in \A$, then 
\[
P^\mathbf{v}_\B(A)=\sum_{B \in \beta_\B} \frac{\sum_{k=1}^N \frac{v_k}{d_k} \sum_{(i,j) \in \Psi^{(k)}_B} A^{(k)}_{i,j}}{\sum_{k=1}^N\frac{|\Psi^{(k)}_B|v_k}{d_k}}B.
\]
\end{theorem}
\begin{proof}
Fix $B \in \beta_\B$.
Since $\Psi^{(k)}_B$ is the set of indices for the nonzero entries of the basis element $B$ contained in the $k$th summand,  we know that 
$$B=\sum_{k=1}^N \sum_{(i,j) \in \Psi^{(k)}_B} E^{(k)}_{i,j}.$$

If $m \in \N$, then  $\mathrm{Tr}(E_{i,j}^*C)=C_{i,j}$ for any $C \in \M_m$  and  $i,j \in [m]$. We use this fact repeatedly in the following calculation. Let $A \in \A$ and observe that 
$$\tau_\mathbf{v}(B^*A)=\tau_\mathbf{v}\left(\left(\sum_{k=1}^N \sum_{(i,j) \in \Psi^{(k)}_B} E^{(k)}_{i,j}\right)^*A\right)=\sum_{k=1}^N \frac{v_{k}}{d_k} \sum_{(i,j)\in \Psi^{(k)}_B}A^{\left(k\right)}_{i,j}.$$

  We also have by Theorem \ref{t:can-basis}
\begin{align*}
\tau_\mathbf{v}(B^*B)&=\tau_\mathbf{v}\left(\left(\sum_{k=1}^N \sum_{(i,j) \in \Psi^{(k)}_B} E^{(k)}_{i,j}\right)^*\left(\sum_{k=1}^N \sum_{(i,j) \in \Psi^{(k)}_B} E^{ (k) }_{i,j}\right)\right)\\
&=\sum_{k=1}^N \frac{v_{k}}{d_k} \sum_{(i,j)\in \Psi^{(k)}_B}1= \sum_{k=1}^N \frac{|\Psi^{(k)}_B| v_{k}}{d_k}.
\end{align*}
Hence, by (3) of Theorem \ref{t:can-basis} and Theorem \ref{t:proj}, we conclude that 
$$P^\mathbf{v}_{\B}(A)=\sum_{B \in \beta_\B}  \frac{\sum_{k=1}^N \frac{v_{k}}{d_k} \sum_{(i,j)\in \Psi^{(k)}_B}A^{\left(k\right)}_{i,j}}{\sum_{k=1}^N \frac{|\Psi^{(k)}_B| v_{k}}{d_k}}B,$$
which completes the proof. \end{proof}
  
We next show how the Frobenius--Rieffel norms recover the Frobenius norm.
\begin{theorem}\label{A_1^n}

For all $A \in \M_n $,   \[\|A\|_{\B_{1^n}}=\|A\|_{F_n},\] where $\|A\|_{F_n}=\frac{1}{\sqrt{n}} \sqrt{\mathrm{Tr}(A^*A)}$ is the Frobenius norm normalized with respect to $I_n$.
\end{theorem}

\begin{proof}
By Theorem  \ref{t:proj-formula}, we have that $$ P_{\mathcal{B}_{1^n} }(A)= \frac{1}{n}\mathrm{Tr}(A)I_n.$$
Therefore, \[
\|A\|_{\B_{1^n} }^2 = \left\|P_{\mathcal{B}_{1^n} }(A^*A)\right\|_{\op} =\left\| \frac{1}{n}\mathrm{Tr}(A^*A)I_n\right\|_{\op}  =\|A\|_{F_n}^2.   \qedhere
\]
\end{proof}

The next two examples show that Frobenius--Rieffel norms   are not generally sub-multiplicative or unitarily invariant.

\begin{example}
Consider the unital C*-subalgebra $\B_{1,1}^2$ of $\M_2$. Let $A= \begin{bmatrix}
1&2\\
2&1
\end{bmatrix}$ and use Theorem \ref{t:proj-formula} to calculate
\begin{align*}
\|A\|_{\B_{1,1}^2}^2=\left\|P_{\B_{1,1}^2}(A^*A)\right\|_{\op} =\left\|P_{\B_{1,1}^2}\left(\begin{bmatrix}
5 & 4 \\
4 & 5 \\
\end{bmatrix}\right)\right\|_{\op} =\left\|\begin{bmatrix}
5 & 0 \\
0 & 5 \\
\end{bmatrix}\right\|_{\op} =5.
\end{align*}
Thus,   $\|A\|_{\B_{1,1}^2} \cdot \|A\|_{\B_{1,1}^2}=5.$ Similarly,   $\|AA\|_{\B_{1,1}^2}=41$, so
\[\|AA\|_{\B_{1,1}^2}>\|A\|_{\B_{1,1}^2}\|A\|_{\B_{1,1}^2}.\] 
\end{example}
\begin{example}
Consider $$A=\begin{bmatrix}
1& 1\\
1 & 1
\end{bmatrix}$$ and the unitary 
$$U=\frac{1}{\sqrt{2}}\begin{bmatrix}
1 & 1\\
1 & -1
\end{bmatrix}.$$
Following similar calculations as the last example, we conclude \[
\|A\|_{\B_{1,1}^2}^2= 2 \neq 4=\left\|U^*AU\right\|_{\B_{1,1}^2}^2.\]
\end{example}

 \section{Equivalence constants for the operator norm}
 
 As discussed in the introduction, it is important to be able to compare the Frobenius--Rieffel norms  with the operator norm.    Theorem  \ref{td:gen-Frob-norm}  says that  
 \[
 \|A\|_{P,\B} =\sqrt{\|P(A^*A)\|_\mathrm{op}}\leq \sqrt{\|A^*A\|_\mathrm{op}}=\sqrt{\|A\|^2_\mathrm{op}}=\|A\|_\mathrm{op}
 \]
for all $A \in \bigoplus_{k=1}^N \M_{d_k}$,  any unital C*-subalgebra $\B \subseteq \A$, and any conditional expectation $P: \A \rightarrow \B$ onto $\B$. This equality is achieved by the identity matrix. Thus, the nontrivial task  is to find a constant $\kappa^+_{P, \B}>0$ such that 
\[
\kappa^+_{P, \B}    \|A\|_\mathrm{op} \leq \|A\|_{P,\B}
\]
for all $A \in \bigoplus_{k=1}^N \M_{d_k}.$

We begin with some general results  and then focus on the case of $\M_n$. Then, we move to the general   case, which is more involved since the Frobenius--Rieffel norms   depend on the underlying    subalgebra and   faithful tracial state.  
We begin with an inequality that allows us to avoid dealing with $A^*A$.

\begin{lemma}\label{l:compare}
Let $\mathcal{B} \subseteq \A=\bigoplus_{k=1}^N \M_{d_k}$ be a unital C*-subalgebra, let $\tau$ be a faithful tracial state on $\A$, and let $\mu \in (0,\infty).$ The following are equivalent.
\begin{enumerate}[leftmargin=*]
\item    $\|C\|_\mathrm{op} \leq \mu \|P^\tau_\A(C)\|_\mathrm{op}$ for all positive  $C \in \A$.
\item  $\|A\|_\mathrm{op}\leq \sqrt{\mu} \|A\|_{\tau, \A}$ for all $A \in \A$.
\end{enumerate}
\end{lemma}

\begin{proof}
We begin with $(1)\implies(2)$. Suppose  $\|C\|_{\op} \leq \mu \|P^\tau_\A (C) \|_{\op}$ for all  positive   $C \in \A$. Then   $\|A^*A\|_{\op} \leq \mu \|P^\tau_\A (A^*A) \|_{\op}$ for all $A \in \A$. Since  $\|A^*A\|_{\op}=\|A\|_\op^2$, we see that  $\|A\|_{\op} \leq \sqrt{\mu} \|A\|_{\tau,\A}.$

For $(2)\implies(1)$, suppose that \[\|A\|_{\op} \leq \sqrt{\mu} \|A\|_{\tau,\A}=\sqrt{\mu} \sqrt{\|P^\tau_\A (A^*A)\|_{\op}}\] for all $A \in  \A$. Then $\|A^*A\|_{\op} =\|A\|_\op^2 \leq \mu \|P^\tau_\A (A^*A) \|_{\op}$. Thus, $\|C\|_{\op} \leq \mu \|P^\tau_\A (C) \|_{\op}$ for all  positive   $C \in  \A$.
\end{proof}

The next  lemma  allows us to extend our results from standard unital C*-subalgebras to all unital C*-subalgebras. The following fact is surprising since, at the end of the last section, we showed that the Frobenius--Rieffel norms are not unitarily invariant in general. Also,   it can be the case that $\|A\|_{\tau, \B} \neq \|A\|_{\tau,\cC}$ for certain $A \in\A$, but the equivalence constants are the same for uniatrily equivalent subalgebras $\B,\cC\subseteq \A$. 

\begin{lemma}\label{l:u-equiv-const}
Let $\tau$ be a faithful tracial state on $\A=\bigoplus_{k=1}^N \M_{d_k}$, let $\B, \cC \subseteq \A$ be unitarily equivalent   (with respect to $\A$) unital C*-subalgebras , and let $\mu \in (0,\infty).$ The following are equivalent.
\begin{enumerate}[leftmargin=*]
\item $\mu \|A\|_{\op} \leq \|A\|_{\tau, \B}$ for all $A \in \A$.
\item   $\mu \|A\|_{\op} \leq \|A\|_{\tau, \cC}$ for all $A \in \A$.
\end{enumerate}
\end{lemma}
\begin{proof}
The argument is symmetric, so we prove only $(1)\implies(2)$. Fix an orthogonal basis $\beta=\{B_1, B_2, \ldots, B_m\}$ for $\B$ with respect to $\tau$. Since $U(\cdot)U^*:\B \rightarrow \cC$ is a linear bijection,     $\beta'=\{UB_1U^*,  UB_2U^*,\ldots, UB_mU^*\}$ is a basis for $\cC$. Furthermore, if $j,k \in [m]$, we have
\begin{align*}
\tau((UB_jU^*)^*UB_kU^*)=\tau(UB_j^*B_kU^*)
=\tau(U^*UB_j^*B_k)=\tau(B_j^*B_k).
\end{align*}
Hence, $\beta'$ is an orthogonal basis for $\cC$ with respect to $\tau$. 

Now  let $A \in \A$.  Theorem \ref{t:proj} implies that
\begin{align*}
P^\tau_\cC(A)&=\sum_{i=1}^m \frac{\tau((UB_iU^*)^*A)}{\tau((UB_iU^*)^*UB_iU^*)}UB_iU^*\\
& = U\left(\sum_{i=1}^m \frac{\tau(UB_i^*U^*A)}{\tau(B_i^*B_i)}B_i\right)U^*\\
& = U\left(\sum_{i=1}^m \frac{\tau(U^*AUB_i^*)}{\tau(B_i^*B_i)}B_i\right)U^* =UP^\tau_\B(U^*AU)U^*.
\end{align*}
For all $A \in  \A$, 
\begin{align*}
\|A\|_{\tau,\cC}^2&=\|P^\tau_\cC(A^*A)\|_{\op}= \|UP^\tau_\B(U^*A^*AU)U^*\|_{\op} = \|P^\tau_\B(U^*A^*AU)\|_{\op}\\
& =\|P^\tau_\B((AU)^*AU)\|_{\op}=\|AU\|^2_{\tau,\B} \geq \mu^2\|AU\|_\op^2=\mu^2 \|A\|^2_\op,
\end{align*}
which completes the proof.
\end{proof}
 We next present a basic lemma about positive  matrices.

\begin{lemma}
If $T=A-B$ for some positive   $A, B \in \M_n$, then $\|T\|_\op  \leq \max\{ \|A\|_\op, \|B\|_\op\}$.
\end{lemma}

\begin{proof}
Since $-\|B\|_{\op}I \leq -B \leq T \leq A  \leq \|A\|_\op I $, it follows that $T - \lambda I_n$ is invertible if $\lambda> \|A\|_{\op}$ or $\lambda < -\|B\|_{\op}$. Thus, the spectrum of the self-adjoint matrix $T$ is contained in the interval $[-\|B\|_{\op},\|A\|_{\op}]$.
\end{proof}

Lemma \ref{l:conj-const} is our main tool in providing equivalence constants. It is motivated by the notion of ``pinching'' in matrix analysis (see \cite{Bhatia00}).
\begin{lemma}\label{l:conj-const}
Let $X \in   \M_n$ be positive. If $P(X)$ is a mean of $n$  unitary conjugates of $X$, $X^\mathsf{T}$ (the transpose of $X$), or $X^*$, one of which is $X$ itself, then $$\|P(X)\|_{\op} \geq \frac{1}{n}\|X\|_{\op}.$$
\end{lemma}

\begin{proof}
Since $X$ is positive, a unitary conjugate of $X$, $X^\mathsf{T}$, or $X^*$ is also positive  (and has the same operator norm as $X$). Suppose that $$P(X)=\frac{1}{n}\sum_{i=0}^{n-1} C_i$$ is a mean of $n$ unitary conjugates $C_i$ of $X$, $X^\mathsf{T}$, or $X^*$ and that $C_0=X$ itself. Since $P(X)$ is   positive, the previous lemma ensures that
$$\|X-P(X)\|_\op = \left\| \frac{n-1}{n} X- \frac{1}{n} \sum_{i=1}^{n-1} C_i \right\|_\op \leq \frac{n-1}{n} \|X\|_\op.$$
Consequently,
\begin{align*}
\|P(X)\|_\op&= \|X+P(X)-X\|_\op\\
&\geq \|X\|_\op -\|X-P(X)\|_\op \geq \|X\|_\op -\frac{n-1}{n} \|X\|_\op=\frac{1}{n}\|X\|_\op,
\end{align*}
which completes the proof. \end{proof}
We first apply this lemma to the following family of unital C*-subalgebras.
\begin{theorem} \label{t:no-rep-blocks}
Let  $\B_\la \subseteq \M_n$ where $\langle \la \vdash n\rangle$ and $\la=(n_1, n_2, \ldots,n_L)$.  

If $X \in \M_n$ is positive, then 
\[
\frac{1}{L }\|X\|_{\op} \leq \|P_{\B_\la}(X)\|_{\op}.
\]

Moreover,   $$\frac{1}{\sqrt{L }}\|X\|_{\op} \leq \|X\|_{\B_\la}$$
for all $X \in \M_n$.
\end{theorem} 

\begin{proof}
 Consider the unitary $U=\bigoplus_{i=1}^{L} \omega^{i}I_{n_{i}}$, where $\omega$ is a primitive $L$th root of unity.  Let $X \in \M_n$. We may write $X$ as blocks in the following way
 \[
 X=\begin{bmatrix}
 X_{n_1} & & & A\\
  & X_{n_2} & & \\
  & & \ddots & \\
  B & & & X_{n_L}
 \end{bmatrix},
 \]
 where $X_{n_k} \in \M_{n_k}$ with $(X_{n_k})_{i,j}=X_{i+n_1+\cdots +n_{k-1},\ j+n_1+\cdots+ n_{k-1}}$ for each $k \in \{1,2,\ldots, L\},$ and $ i,j \in \{1,2,\ldots, n_k\}$, and $A$ and $B$ denote the remaining entries of $X$.  
By Theorem \ref{t:proj-formula}, it follows that 
\[P_{\B_\la}(X)= \begin{bmatrix}
 X_{n_1} & & & 0\\
  & X_{n_2} & & \\
  & & \ddots & \\
  0 & & & X_{n_L}
 \end{bmatrix}.
\]
On the other hand, a direct computation shows that 
\[\frac{1}{L}\sum_{i=0}^{L-1}U^{i}XU^{*i}= \begin{bmatrix}
 X_{n_1} & & & 0\\
  & X_{n_2} & & \\
  & & \ddots & \\
  0 & & & X_{n_L}
 \end{bmatrix}.\] 
 Hence,  $P_{\B_\la}(X)=\frac{1}{L}\sum_{i=0}^{L-1}U^{i}XU^{*i}$.  By Lemma \ref{l:conj-const}, we have that $\|P_{\B_\la}(X)\|_{\op} \geq (1/L)\|X\|_{\op}$. 

By Lemma \ref{l:compare}, we have  $$\frac{1}{\sqrt{L }}\|X\|_{\op} \leq \|X\|_{\B_\la}$$ for all $X \in \M_n$. 
\end{proof}

We can now use the ideas from Theorem \ref{t:no-rep-blocks} to calculate equivalence constants for a  subalgebra of the form $\B_\la$ for arbitrary $\la$ (Definition \ref{n:s-full}). The idea of the proof is as follows. Assume we want to project a matrix of the form
\[ X=\begin{bmatrix}
A_{1,1} & A_{1,2} & A_{1,3} \\
A_{2,1} & A_{2,2} & A_{2,3}\\
A_{3,1}& A_{3,2}& A_{3,3}
\end{bmatrix}\]
onto the subalgebra of matrices of the form \[
\begin{bmatrix}
B & 0 & 0 \\
0 & B & 0\\
0 & 0 & C
\end{bmatrix}.\]
We can do this in two steps. First project $X$ onto 
\[
Y=\begin{bmatrix}
A_{1,1} & 0 & 0 \\
0 & A_{2,2} & 0\\
0& 0& A_{3,3}
\end{bmatrix},
\]
 which is the setting of Theorem \ref{t:no-rep-blocks}.  Then  project $Y$ onto 
 \[
 \begin{bmatrix}
M & 0 & 0 \\
0 & M & 0\\
0& 0& A_{3,3}
\end{bmatrix}.
 \]
 The proof of the next theorem shows how we can represent this final projection using a mean of unitary conjugates, which   allows us to utilize Lemma \ref{l:conj-const} as done in the proof of  Theorem \ref{t:no-rep-blocks}. The reason for this two-step approach is that it does not seem feasible to represent the projection directly onto the desired subaglebra as a mean of unitary conjugates. 

\begin{theorem}\label{t:gen-eq-const}
Consider $\B_\la \subseteq \M_n$ such that $\langle \la \vdash n\rangle$, where $\la=(n_1^{m_1}, n_2^{m_2}, \ldots, n_L^{m_L})$. Set $r=\sum_{i=1}^{L } m_i$ and $\ell = \mathrm{lcm}\{m_1,m_2, \ldots m_{L }\}$. 
If $X \in \M_n$ is positive,  then    
$$\|P_{\B_\la}(X)\|_{\op} \geq\frac{1}{r\ell}\|X\|_{\op}.$$

Moreover, 
\[
\|X\|_\op \geq  \frac{1}{\sqrt{r\ell}}\|X\|_\op
\]
for all $X \in \M_n.$
\end{theorem}

\begin{proof}
  We write $P_{\B_\la}$ as the composition of two maps.  For each $i \in [ r]$, set
\begin{equation}\label{eq:ind-sum}e_i=
\begin{cases}
n_1 & \text{if } 1 \leq i \leq m_1,\\
n_j & \text{if }  2\leq j \leq L \text{ and } 1+\sum_{p=1}^{j-1} m_p \leq i \leq \sum_{p=1}^j m_p,
\end{cases}\end{equation}
that is,   $e_1=n_1, \ \ e_2=n_1, \ \ldots, \ \  e_{m_1}=n_1$, and 
\[e_{m_1+1}=n_2,\  \ e_{m_1+2}=n_2, \ \ldots, \ \ e_{m_1+m_2}=n_2,\]
etc.  Now
 set   $\la'=(e_1, e_2 \ldots, e_r)$ and  note that  $\langle \la' \vdash n\rangle$. 
By Theorem \ref{t:no-rep-blocks}, we have $\|P_{\B_{\la'}}(X)\|_{\op} \geq (1/r)\|X\|_{\op}$ for all positive   $X \in \M_n$. 

For each $i \in [L],$  let  $V_{j,i}$ to be the $n_i m_i \times n_i m_i$ circulant matrix with all zeros in the first row, except for a $1$ in the $(1+j n_i)$th position for $j\in 0, 1, \ldots, m_i-1.$ Then we define $V_j=\bigoplus_{i=1}^k V_{(j \mod m_i),i}$ for $j=0,1,  \ldots, \ell-1$ where $\ell = \mathrm{lcm}\{m_1, m_2 \ldots, m_L\}.$  For any positive  $X \in \M_n$,  define  
$$Q(X)=\frac{1}{l}\sum_{j=1}^{l-1}V_j X V_j^*.$$ 


By Lemma \ref{l:conj-const} $\|Q(X)\|_{\op} \geq (1/\ell)\|X\|_{\op}$ for all positive   $X \in \M_n$. Then, a direct computation along with  Theorem \ref{t:proj-formula} provides   that  $P_{\B_\la}(X)=Q(P_{\B_{\la'}}(X))$, which gives us $$\|P_{\B_\la}(X)\|_{\op} \geq \frac{1}{r\ell} \|X\|_{\op},$$ for any positive  $X \in \M_n.$ The rest follows from   Lemma \ref{l:compare}. \end{proof}
\begin{example}
We calculate the values of $r,\ell$ for the following subalgebras of $\M_5.$ 

   For $\B^5_{3,2}$, we have $r=1+1=2$ and $\ell=\mathrm{lcm} \{1,1\}=1$. Thus $r\ell=2$.

For $\B^5_{2^2,1}$, we have $r=2+1=3$ and $\ell=\mathrm{lcm} \{2,1\}=2$. Thus $r\ell=6$.

For $\B^5_{2,1^2,1}$, we have $r=1+2+1=4$ and $\ell=\mathrm{lcm} \{1,2,1\}=2$. Thus $r\ell=8$.

For $\B^5_{2,1^3}$, we have $r=1+3=4$ and $\ell=\mathrm{lcm} \{1,3\}=3$. Thus $r\ell=12$.

We also note that for the subalgebra $\B^4_{1^3,1} \subseteq \M_4$, we have $r=3+1=4$,  $\ell=\mathrm{lcm}\{3,1\}=3$, and $r\ell=12.$
\end{example}
 
Thus, combining Theorem \ref{t:gen-eq-const}  with Lemma \ref{l:u-equiv-const} and Theorem \ref{t:all-sub}, we have found equivalence constants   for Frobenius--Rieffel norms constructed from any unital C*-subalgebra of $\M_n$ built from natural structure (the dimensions of the terms of the block diagonals of the given subalgebra).

 Table \ref{table:1}  outlines  the equivalence constants for all unital *-subalgebras of $\M_n$ for $1 \leq n \leq 5$. The second column contains   equivalence constants suggested by brute force using software (this was done by making   software calculate the operator and Frobenius--Rieffel norms of many  matrices, and then making a guess), which we think might be the sharp  equivalence constants. The third column contains the theoretical equivalence constant found in  Theorems  \ref{t:no-rep-blocks},   \ref{t:gen-eq-const}.  Our goal in this paper is not to find the sharp equivalence constants, but just explicit ones that afford us some continuity results as mentioned in the first section. It remains an open question to find the sharp constants, and this table suggests that we may have found the sharpest constants in some cases.

\begin{center}
\begin{table}[h]
\begin{tabular}{|c|c|c|} 
 \hline
 Algebra & Guess of Sharp Equiv. Const. & Theorem \ref{t:gen-eq-const} Equiv. Const. \\ 
 \hline
  $\B^3_{2,1}$ & $1/\sqrt{2}$ & $1/\sqrt{2}$ \\  
 $\B^3_{1^2, 1}$ & $1/\sqrt{3}$ & $1/\sqrt{6}$ \\  
  $\B^4_{2,2}$ & $1/\sqrt{2}$ & $1/\sqrt{2} $\\  
  $\B^4_{2^2}$ & $1/\sqrt{4}$ & $1/\sqrt{4} $\\  
 $\B^4_{2,1,1}$ & $1/\sqrt{3}$ & $1/\sqrt{3}  $\\  
$\B^4_{2,1^2}$ & $1/\sqrt{3}$ & $1/\sqrt{6}$ \\  
$\B^4_{1^3,1}$ & $1/\sqrt{4}$ & $1/\sqrt{12}$ \\  
$ \B^4_{1^2,1,1}$ & $ 1/\sqrt{4}$ & $1/\sqrt{8} $\\  
 $ \B^5_{3,2}$ & $1/\sqrt{2}$ & $1/\sqrt{2} $\\ 
 $ \B^5_{2,2,1} $& $1/\sqrt{3}$ & $1/\sqrt{3} $\\  
 $\B^5_{2^2,1}$ & $1/\sqrt{4}$ & $1/\sqrt{6}$ \\  
  $ \B^5_{3,1,1} $& $1/\sqrt{3}$ & $1/\sqrt{3}  $\\  
   $  \B^5_{3,1^2} $& $1/\sqrt{3}$ & $1/\sqrt{6}  $\\  
 $\B^5_{2,1,1,1}$ & $1/\sqrt{4} $& $1/\sqrt{3}$\\  
 $\B^5_{2,1^3}$ & $1/\sqrt{4}$ & $1/\sqrt{12}$ \\  
 $\B^5_{2,1^2,1}$ & $1/\sqrt{4}$ & $1/\sqrt{8}$\\  \hline  
 \end{tabular} 
 \caption{  
 Theorem \ref{t:gen-eq-const} equivalence constants and guesses of sharp equivalence constants}
\label{table:1}
\end{table}
 \end{center}

\subsection{The general case}
We now study the   case of $\bigoplus_{k=1}^N \M_{d_k}$, which is much more involved for two main reasons. First, as seen in Example \ref{e:d-sum-basis}, the canonical basis elements for standard unital C*-subalgebras of $\bigoplus_{k=1}^N \M_{d_k}$ can have non-zero terms in multiple summands, which requires more bookkeeping than the previous section. Second, the Frobenius--Rieffel norms now vary on an extra parameter: the faithful tracial state. In the $\M_n$ case, the only faithful tracial state is $\frac{1}{n} \mathrm{Tr}$, so this was not an issue. For instance, consider $\M_2 \oplus \M_2$ and
the subalgebra
\[
\B=\{ \diag(\mu,\nu)\oplus \diag(\mu,\mu) : \mu,\nu \in \C\}.
\] To build a Frobenius--Rieffel norm on $\M_2 \oplus \M_2$ with respect to $\B$, we also need a faithful tracial state on $\M_2\oplus \M_2$. We could take $\tau_{(1/4,3/4)}$ on $\M_2 \oplus \M_2$ (see Example \ref{e:trace}). Hence, taking into account the expression for the associated conditional expectation of Theorem \ref{t:proj-formula}, we need to keep track of how the coefficients $1/4$ and $3/4$ impact the construction of the Frobenius--Rieffel norm since $\mu$ appears in both summands. Thus, we cannot simply view $\B$ as a subalgebra of $\M_4$ and proceed to use the previous section since we would lose track of the weights since $\M_4$ has a unique faithful tracial state.  The following definition environment allows us to collect all the terms that we use to find our equivalence constants in this much more involved setting. We note that we generalize the
constants $r, \ell$ from Theorem  \ref{t:gen-eq-const}.


\begin{definition}\label{n:gen-form}

Let $\B\subseteq \A=\bigoplus_{k=1}^N \M_{d_k}$ be a standard unital C*-subalgebra, where  for each  $k \in [N]$, we have    $
p_k(\B)=\B^{d_k}_{\la_k}
$ 
with $\langle \la_k \vdash d_k\rangle$. We denote $\mathbf{m}_{\la_k}=(m_{k,1}, m_{k,2}, \ldots, m_{k,L_k})$ and $\mathbf{n}_{\la_k}=(n_{k,1}, n_{k,2}, \ldots, n_{k,L_k})$.
 
 Next, we collect the data we need associated to a given faithful tracial state. Let $\mathbf{v}=(v_1, v_2, \ldots, v_N)\in (0,1)^N$ such that $\sum_{k=1}^N v_k=1$, and  let $\{b_1,b_2,  \ldots, b_M\}$ be the canonical  orthogonal basis for $\B$ given by matrix units.
 
 Define:
 \begin{enumerate}
 \item $\ell=\mathrm{lcm} \left\{m_{k,i} : k\in [N], i \in [L_k]\right\}$,
 \item $r=\mathrm{lcm}\{r_1, r_2 \ldots, r_N\}$,  where $r_k $ is the number of blocks of $\B$ in the $k$th summand of $\A$ for each $k \in [N]$, 
 \item  $m=\mathrm{lcm}\{m_{b_1}, \ldots, m_{b_M}\}$, where $m_{b_i}$ is the number of nonzero entries of the basis element $b_i$ for each $i \in [M]$,
 \item $\alpha=\min\left\{\frac{v_{k}}{d_k} : k \in [ N] \right\},$ and
 \item $\gamma=\max\left\{\sum_{k=1}^N  \frac{\rho_{k,i} v_{k}}{d_k} :   i \in [M]\right\}$, where $\rho_{k,i}$ is the number of times there is a nonzero entry of $b_i$ in the $k$th summand of $\A$ for each $i \in [M]$ and $k \in [N].$
 \end{enumerate}
\end{definition}
 First, we tackle the subalgebras of the form $\cC_\B$ in Definition \ref{d:s-sub}, which recovers Theorem \ref{t:gen-eq-const} when $N=1$.
 
 \begin{theorem}
 \label{thm:gen_eq_simple}
 Consider $\A=\bigoplus_{k=1}^N \M_{d_k}$. For each $k \in [N]$, consider $ \B_{\la_k} \subseteq    \M_{d_k} $ 
such that $\langle \la_k\vdash d_k\rangle$. Set
\[
\B=\bigoplus_{k=1}^N \B_{\la_k}.
\] 
Let $\mathbf{v}=(v_1, v_2, \ldots, v_N)\in (0,1)^N$ such that $ \sum_{k=1}^N v_k=1$. 
 If $X \in \A$ is positive, then 
$$\|P^\mathbf{v}_\B(X)\|_{\op} \geq \frac{1}{r\ell}\|X\|_{\op},$$
and,  moreover,
\[
\frac{1}{\sqrt{r\ell}}\|X\|_{\op} \leq \|X\|_{\mathbf{v},\B}
\]   
for all $X \in \A$.
\end{theorem}

\begin{proof}
For each  $B\in \beta_\B$, let $k_B \in [N]$ be the   summand where $B$ has a non-zero entry.  Theorem \ref{t:proj-formula} implies that 
$$P^\mathbf{v}_{\B}(A)=\sum_{B \in \beta_\B} \frac{\sum_{(i,j) \in \Psi_{B,k_B}}A_{i,j}^{\left({k_B}\right)}}{|\Psi_{B,k_B}|}B$$
for all $A \in \A. $

We recover $P^\mathbf{v}_{\B}$ using a mean of unitary conjugates in two steps. Let $k \in [N]$. Suppose the $i$th block of $\B_{\la_k}$ has dimension $(e^{(k)}_{i})^2$ (see Expression \eqref{eq:ind-sum}). Set $\la_k'=( e^{(k)}_{1}, e^{(k)}_{2}, \ldots, e^{(k)}_{r_{k}})$  and note that $\langle \la_k' \vdash d_k\rangle$.  Then, let \[U^{(k)}=\bigoplus_{i=1}^{r_k} \omega^{i}I_{e^{(k)}_{{i}}},\] where $\omega$ is a primitive $r_k$th root of unity. 

Note that $U= (U^{(1)}, \ldots, U^{(N)} )$  is unitary as each $U^{(k)}$ is unitary. We then define $P_1:
  \bigoplus_{k=1}^N \A \rightarrow \bigoplus_{k=1}^N \B_{\la_k'}$   by $$P_1\left(X\right)=\bigoplus_{k=1}^N \frac{1}{r}\sum_{i=0}^{r-1}\left(U^{(k)}\right)^{ i \text{ mod } r_k}X^{(k)}\left(\left(U^{(k)}\right)^*\right)^{i \text{ mod } r_k},$$
  where $i \text{ mod } r_k \in \{0, 1, \ldots, r_k-1\}.$ 
   By Lemma \ref{l:conj-const}, we have  $\|P_1(X)\|_{\op} \geq (1/r)\|X\|_{\op}$.

Using the convention for $\mathbf{m}_{\la_k}, \mathbf{n}_{\la_k}$ in Notation \ref{n:gen-form}, we then define, for $k\in [N],i \in [L_k]$, the matrix     $V_{k,j,i}$ to be the $n_{k,i} m_{k,i} \times n_{k,i} m_{k,i}$ circulant matrix with all zeros in the first row, except for a $1$ in the $(1+j n_{k,i})$th position for  $0 \leq j \leq m_{k,i}-1$. 
Set $V_j^{(k)}=\bigoplus_{i=1}^{L_k} V_{k, (j \text{ mod } m_{k,i}),i}$ for $j=0, \ldots, \ell-1$, and let $$V_j=\left(V_j^{(1)}, \ldots, V_j^{(N)}\right).$$ Then  define $P_2:\bigoplus_{k=1}^N \B_{\la_k'} \rightarrow  \B$ by 
$$P_2(X)=\frac{1}{\ell}\sum_{j=0}^{\ell-1}V_j X V_j^*.$$ 
Since $V_0 =I$, we know $\|P_2(X)\|_{\op} \geq (1/\ell)\|X\|_{\op}$ by Lemma \ref{l:conj-const}. 
We also have  that $P^\mathbf{v}_\B =P_2\circ P_1$ by construction. Hence $$\|P^\mathbf{v}_\B(X)\|_{\op} \geq \frac{1}{r\ell} \|X\|_{\op},$$ which completes the proof by   Lemma \ref{l:compare}. \end{proof}

The values of $\mathbf{v}=(v_1, v_2, \ldots, v_N)\in (0,1)^N$ do not appear in the calculations above. This makes sense because the case of Theorem \ref{thm:gen_eq_simple} is essentially the case when $N=1$ since the non-zero entries of a basis element do not appear in multiple summands, and so the different coordinates of $v$ do not appear and we   simply work  with $ \sum_{k=1}^N v_k=1$.  Thus, we now move towards the case when the non-zero entries of our basis elements can appear in multiple summands,  such as in Example \ref{e:d-sum} and as in the subalgebras defined before Theorem \ref{e:c-frac}.  To provide intuition for the following proof, we revisit the example at the beginning of the section.  Consider  
 $\M_2 \oplus \M_2$ and
the C*-subalgebra
\[
\B=\{ \diag(\mu,\nu)\oplus \diag(\mu,\mu) : \mu,\nu \in \C\}.
\]
The first step of the following proof  is to project an $A \oplus B \in \M_2 \oplus \M_2$ onto an element of the form $\diag (a,b) \oplus \diag(c,d) \in \M_2 \oplus \M_2$. Next, in order to project $\diag (a,b) \oplus \diag(c,d)$ into $\B$, we   view $\diag (a,b) \oplus \diag(c,d)$ as $\diag(a,b,c,d) \in \M_4$ and we   view elements of $\B$ as $\diag(\mu,\nu,\mu,\mu)$.   Then we   use a mean of unitary conjugates in $\M_4$ to project $\diag(a,b,c,d)$ to an element of the form $\diag(\mu,\nu,\mu,\mu)$, which  is an element in $\B$. To form the unitaries, begin with $W_1=I_4$. Next, since the  $(1,1)$-entry in $\diag(\mu,\nu,\mu,\mu)$ repeats in the $(3,3)$-entry and $(4,4)$-entry,   we permute the first, third, and fourth column of $W_1=I_4$ two times to get two more unitaries \[
W_2=\begin{bmatrix}
0 & 0 & 1 &0 \\
0 & 1 & 0 & 0 \\
0 & 0 & 0 & 1\\
1 & 0 & 0 & 0
\end{bmatrix} \quad \quad \text{ and } \quad \quad 
W_3=\begin{bmatrix}
0 & 0 & 0 &1 \\
0 & 1 & 0 & 0 \\
1 & 0 & 0 & 0\\
0 & 0 & 1 & 0
\end{bmatrix}.
\]
If we permute these columns one more time,  then we obtain $I_4$. 
Note that \[\sum_{i=1}^3 W_i \diag(a,b,c,d)W_i^* \in \B.\]  Using Definition \ref{n:gen-form}, note that $m=\mathrm{lcm}\{3,1\}=3$ since the standard basis elements of $\B$ are $\diag(1,0,1,1)$ and $\diag(0,1,0,0)$.

\begin{theorem}
 \label{thm:gen_gen_eq}
  Let $\mathbf{v}=(v_1, v_2, \ldots, v_N)\in (0,1)^N$ such that $ \sum_{k=1}^N v_k=1$.  Let $\B$ be a standard unital C*-subalgebra of $\A=\bigoplus_{k=1}^N\M_{d_k}$. If $X\in \A$ is positive, then   
$$\frac{\alpha}{r\ell m \gamma}\|X\|_{\op} \leq \|P_\B^\mathbf{v}(X)\|_{\op},$$
and, moreover,
\[\frac{\sqrt{\alpha}}{\sqrt{ r\ell m\gamma}}\|X\|_{\op} \leq \|X\|_{\mathbf{v},\B}\]
for all $X \in \A$.
\end{theorem}

\begin{proof} 
For $\mathcal{C}_\B$ as defined in Definition \ref{d:s-sub}, we have $\|P^\mathbf{v}_{\mathcal{C}_\B}(X)\|_{\op} \geq (1/(r\ell))\|X\|_{\op}$ for positive   $X \in \A$ by Theorem \ref{thm:gen_eq_simple}.

We then define $$P'(X)=\bigoplus_{k=1}^N \frac{v_{k}}{a_k}P_{\mathcal{C}_\B}^\mathbf{v}(X)^{(k)},$$ which gives us $\|P'(X)\|_{\op} \geq \frac{\alpha}{r\ell} \|X\|_{\op}$ for  all positive   $X \in \bigoplus_{k=1}^N\M_{d_k}$.

Suppose $e_k^2$ is the dimension of the $k$th block of $\B$ and $b$ is the total number of blocks of $\B$. For the following, we view $\B$ and $\A$ as subalgebras of $\M_{d} $, where $d= \sum_{k=1}^N d_k$.
Let  $$W_1= \bigoplus_{k=1}^b I_{e_k}=I_d.$$ 
We construct $W_2$ by permuting the blocks of $W_1$ in the following way.  If the $k$th block of $\B$ is   not repeated, then fix $I_{e_k}$. Next, assume that the  $k$th block  of $\B$ is repeated and that the $k$th block is the first position this repeated block appears. Assume that the $j$th block is the next block to the right that the the $k$th block is repeated. Then   $I_{e_k}$ stays  in the same rows it occupied in $W_1$, but its columns permute to the columns (in $\M_{d} $) of the $j$th block in $\B$. If the $j$th block is repeated, then repeat this process with $I_{e_j}$. However, if  the $j$th block is not repeated, then permute the columns $I_{e_j}$ to the columns of the $k$th block. Continue in this way until all blocks are either permuted or fixed depending on repetition or lack thereof, which gives us $W_2$.  Repeat this process to make $W_3, W_4,\ldots, W_m$, where $m$ is defined in  (3) of Definition \ref{n:gen-form} (see the example before the statement of the theorem).  Note that $W_{m+1}=I_d$.  Define
$f:\mathcal{C}_\B \rightarrow  \B$   by $$f(X)=\frac{1}{m} \sum_{i=1}^m W_i X W_i^*,$$ which satisfies $$\|f(P'(X))\|_{\op} \geq \frac{\alpha}{r\ell m }\|X\|_{\op}$$ for all positive  $X \in  \cC_\B$ by Lemma \ref{l:conj-const}. 

Finally, by Theorem \ref{t:proj-formula} and a direct computation, we have that 
\[
\|P^\mathbf{v}_\B(X)\|_{\op}=\frac{1}{\gamma}\|f(P'(X))\|_{\op}.  
\]We conclude that   
$$\|P^\mathbf{v}_\B(X)\|_{\op} \geq \frac{\alpha}{r\ell m \gamma}\|X\|_{\op}$$ for all positive  $X \in \A$. Lemma \ref{l:compare} completes the proof. \end{proof}

We can use the previous theorem to find   equivalence constants for all unital *-subalgebras $\B \subseteq \bigoplus_{k=1}^N \M_{d_k}$ by Lemma \ref{l:u-equiv-const}.
 
 \section{An application to Effros--Shen algebras}
 To finish, we now apply our main result to the finite-dimensional C*-algebras
in the inductive sequence used by Effros and Shen in the construction
of their AF algebras from the continued fraction expansion of irrational
numbers  \cite[Section VI.3]{Davidson}, \cite{Effros80b}. These algebras provide a suitable example to test our results. Indeed, in \cite{Aguilar-Latremoliere15}, it was shown that the Effros--Shen algebras vary continuously with respect to their irrational parameters in a noncommutative analogue to the Gromov--Hausdorff distance, called the dual Gromov--Hausdorff propinquity \cite{Latremoliere13b}. A crucial part of this result is the fact that each Effros--Shen algebra comes equipped with a unique faithful tracial state and that the faithful tracial states themselves vary continuously  with respect to the irrational parameters. Therefore, to test our results in the previous section, we will see that  for the Frobenius--Rieffel norms that are  built from  these faithful tracial states,   this continuity passes through to the equivalence constants. This   further displays how far-reaching  the information of the irrational parameters appears in structures related to the Effros--Shen algebras.   
 
 First, given an irrational $\theta \in (0,1)$, the Effros--Shen algebras are built using the   continued fraction expansion of $\theta$. The continuity results in   \cite{Aguilar-Latremoliere15} were established using the {\em Baire space}, a metric space that is homeomorphic to $(0,1)\setminus \Q$ with its usual topology. The Baire space  is the set of positive integer sequences, which is in one-to-one correspondence with $(0,1)\setminus \Q$ via the continued fraction expansion, equipped with the Baire metric. We begin reviewing continued fractions and the Baire space. Background on continued fractions can be found in many introductory number theory texts, such as  \cite{Hardy38}.

 Let $\theta \in \R$ be irrational. There exists a unique sequence of   integers $(r^\theta_n)_{n \in \N_0}$ (where $\N_0=\{0\}\cup \N$) with $r^\theta_n>0$ for all $n \in \N$ such that 
 \[
 \theta =\lim_{n \to \infty} r_0^\theta +\cfrac{1}{r^\theta_1 + \cfrac{1}{r^\theta_2 + \cfrac{1}{r^\theta_3 +\cfrac{1}{\ddots+\cfrac{1}{r^\theta_n}}}}}.
 \]
 When $\theta \in (0,1)$, we have that $r^\theta_0=0$. The sequence $(r^\theta_n)_{n \in \N_0}$ is called the {\em continued fraction expansion of $\theta$}. 
 
 To define the Baire space, first let $\mathcal{N}$ denote the set of positive integer sequences. The  Baire metric $d_B$ on $\mathcal{N}$ is defined by 
 \[
 d_B(x,y)=\begin{cases}
 0 & \text{ if  } x=y,\\
 2^{-\min \{ n \in \N : x_n\neq y_n\}} & \text{ if } x \neq y.
 \end{cases}
 \]  The  metric space $(\mathcal{N}, d_B)$ is   the {\em Baire space}.  In particular, the distance in the Baire metric between two positive integer sequences   is less than   $2^{-n}$  if and only if their terms  agree up to $n$. We now state the following well-known result in the descriptive set theory literature.
 \begin{proposition}[\!{\cite[Proposition 5.10]{Aguilar-Latremoliere15}}]\label{p:baire}
 The map
 \[
 \theta \in (0,1)\setminus \Q \mapsto (r^\theta_n)_{n \in \N} \in \mathcal{N}
 \]
 is a homeomorphism with respect to the usual topology on $\R$ and the Baire metric.
 \end{proposition}
 Thus, convergence of a sequence of irrationals    to an irrational  in the usual topology on $\R$ can be expressed in terms of their continued fraction expansions using the topology induced by the Baire metric.

 Next, we  define the finite-dimensional  C*-subalgebras of the Effros--Shen algebras.
 For each $n \in \N$, define \[
p_0^\theta=r_0^\theta, \quad p_1^\theta=1 \quad \text{ and } \quad q_0^\theta=1, \quad q_1^\theta=r^\theta_1
\]
and   set 
\[
p_{n+1}^\theta=r^\theta_{n+1} p_{n}^\theta+p_{n-1}^\theta
\]
and 
\[
q_{n+1}^\theta= r^\theta_{n+1} q_{n}^\theta+q_{n-1}^\theta.
\]
The sequence $\left(p_{n}^\theta/q_{n}^\theta\right)_{n \in \mathbb{N}_0}$ of {\em convergents} $p^\theta_n/q^\theta_n$ converges to $\theta$. In fact, for each $n \in \N$, 
\[
 \frac{p_n^\theta}{q_n^\theta}=r_0^\theta +\cfrac{1}{r^\theta_1 + \cfrac{1}{r^\theta_2 + \cfrac{1}{r^\theta_3 +\cfrac{1}{\ddots+\cfrac{1}{r^\theta_n}}}}}.
\]
 
 We now define the C*-algebras with which we endow  Frobenius--Rieffel norms. We set $\mathcal{A}_{\theta,0}=\mathbb{C}$ and, for each $n \in \mathbb{N}_0$, we set
\[
\mathcal{A}_{\theta,n}=\M_{q_n^\theta} \oplus \M_{q_{n-1}^\theta}.
\]
For the subalgebras, define 
 \begin{equation}\label{eq:theta-alg}
\alpha_{\theta,n}:A\oplus B \in \mathcal{A}_{\theta,n} \mapsto   \diag\left(A, \ldots, A,B \right)  \oplus A \in \mathcal{A}_{\theta,n+1},
\end{equation}
where there are $r^\theta_{n+1}$ copies of $A$ on the diagonal in the first summand of $\mathcal{A}_{\theta,n+1}$. This   is a unital *-monomorphism by construction.  For $n=0$,
\[\alpha_{\theta, 0}: \lambda \in  \mathcal{A}_{\theta,0} \mapsto   \diag(\lambda, \ldots, \lambda)  \oplus \lambda\ \in \mathcal{A}_{\theta,1}.
\]
For each $n \in \mathbb{N}_0$, set 
\[
\mathcal{B}_{\theta,n+1}=\alpha_{\theta,n}(\mathcal{A}_{\theta,n}),
\]
which is a standard unital C*-subalgebra of $\mathcal{A}_{\theta,n+1}$.

To complete the construction of the Frobenius--Rieffel norm, we need to define a faithful tracial state. We begin with \[t(\theta,n)=(-1)^{n-1}q_n^\theta  (\theta q_{n-1}^\theta -p_{n-1}^\theta ) \in (0,1).\]
Then  set
\[ \mathbf{v}_{\theta,n}=(t(\theta,n),1-t(\theta,n)),\]
so   for all $(A,B) \in \mathcal{A}_{\theta,n}$, we have
\[
\tau_{v_{\theta,n}}(A,B)=t(\theta,n)\frac{1}{q_n^\theta}\mathrm{Tr}(A)+(1-t(\theta,n))\frac{1}{q_{n-1}^\theta}\mathrm{Tr}(B).
\]
For each $n \in \mathbb{N}$, the Frobenius--Rieffel norm on $\A_{\theta, n}$ associated to $ \mathbf{v}_{\theta,n}$ and to the unital C*-subalgebra $\B_{\theta,n}$ is denoted by 
\[ \|\cdot \|_{\mathbf{v}_{\theta,n},\mathcal{B}_{\theta,n}}.
\]  
 
 We conclude the paper with the following theorem, which shows that the equivalence constants we found in this paper are natural in the sense that they reflect the established continuity of the Effros-Shen algebras with respect to their irrational parameters.
 \begin{theorem}\label{e:c-frac}
 Let $\theta \in (0,1) \setminus \Q$ and $N \in \N$. Then 
 \[
 \sqrt{\frac{\theta q^\theta_N - p^\theta_N}{\left( \theta q^\theta_{N-2}-p^\theta_{N-2}\right)r^\theta_N (r^\theta_N+1)^2}}\cdot \|a\|_\mathrm{op} \leq \|a\|_{\mathbf{v}_{\theta,N},\mathcal{B}_{\theta,N}}\leq \|a\|_\mathrm{op}
 \]
 for all $a \in \A_{\theta, N}$. If $(\theta_n)_{n \in \N}$ is a sequence in $(0,1) \setminus \Q$ converging to some $\theta \in (0,1) \setminus \Q$, then 
 \[
 \lim_{n \to \infty} \frac{\theta_n q^{\theta_n}_N - p^{\theta_n}_N}{\left( \theta q^{\theta_n}_{N-2}-p^{\theta_n}_{N-2}\right)r^{\theta_n}_N (r^{\theta_n}_N+1)^2} =  \frac{\theta q^\theta_N - p^\theta_N}{\left( \theta q^\theta_{N-2}-p^\theta_{N-2}\right)r^\theta_N (r^\theta_N+1)^2}.
 \]
 \end{theorem}
 \begin{proof}
 First, we gather the necessary information from the canonical basis of $\B_{\theta, n}$ given by matrix units. Let \[\beta_{\theta,n}=\left\{b_1, \ldots, b_{(q_{n-1}^\theta)^2}\right\}\] be the set of basis elements that span elements of the form $\alpha_{\theta,n-1}(A,0)\in \B_{\theta,n}$. Let
\[
\beta'_{\theta,n}=\left\{b_{(q_{n-1}^\theta)^2+1}, \ldots, b_{(q_{n-1}^\theta)^2+(q_{n-2}^\theta)^2}\right\}
\]
be the set of basis elements that span elements of the form $\alpha_{\theta,n-1}(0,B)\in \B_{\theta,n}$. Note for $n=1$, we have $\beta'_{\theta,n}=\emptyset$. Thus, the canonical basis for $\B_{\theta,n}$ is
\[
\beta_{\B_{\theta,n}}=\beta_{\theta,n}\cup \beta'_{\theta,n}. 
\]
Using Definition \ref{n:gen-form}, we have \[\ell(\theta,n)=\mathrm{lcm}\{r^\theta_{n},1,1\}=r^\theta_{n }\] and \[r(\theta,n)=\mathrm{lcm}\{r^\theta_{n}+1,1\}=r^\theta_{n}+1.\]
Next
\[
m(\theta,n)=\mathrm{lcm}\{r^\theta_{n}+1, 1\}=r^\theta_{n}+1
\]
and 
\begin{align*}
\alpha(\theta,n)&=\min \left\{ (-1)^{n-1}\left(\theta q^\theta_{n-1}-p^\theta_{n-1}\right), (-1)^n \left( \theta q^\theta_n - p^\theta_n\right)\right\}\\
&= (-1)^n \left( \theta q^\theta_n - p^\theta_n\right),
\end{align*}
where the second term is given at the end of the proof of \cite[Lemma 5.5]{Aguilar-Latremoliere15}, and finally
\begin{align*}
\gamma(\theta,n) 
&=\max \Big\{r^\theta_n \cdot (-1)^{n-1}\left(\theta q^\theta_{n-1}-p^\theta_{n-1}\right)+(-1)^n \left( \theta q^\theta_n - p^\theta_n\right),\\
& \quad \quad \quad \quad \quad  (-1)^{n-1}\left(\theta q^\theta_{n-1}-p^\theta_{n-1}\right) \Big\}\\
&=r^\theta_n \cdot (-1)^{n-1}\left(\theta q^\theta_{n-1}-p^\theta_{n-1}\right)+(-1)^n \left( \theta q^\theta_n - p^\theta_n\right)\\
& = (-1)^{n-2}(\theta q^\theta_{n-2}-p^\theta_{n-2}).
\end{align*}
Thus, we conclude that the equivalence constant of Theorem  \ref{thm:gen_gen_eq} is
\begin{equation}\label{eq:es-cont}\sqrt{\frac{\theta q^\theta_n - p^\theta_n}{\left( \theta q^\theta_{n-2}-p^\theta_{n-2}\right)r^\theta_n (r^\theta_n+1)^2}}
.\end{equation}
 
Next, by Proposition \ref{p:baire}, for fixed $n \in \N$,   there exists $\delta>0$ such that if $\eta \in (0,1)\setminus \Q$ and $|\theta-\eta|<\delta$, then $r_m^\theta=r_m^\eta$  for all $m \in \{0, \ldots, n+1\}$, and thus the same holds for $p_m^\theta=p_m^\eta$ and $  q_m^\theta=q_m^\eta$.   In particular, for  irrational $\theta$,     \eqref{eq:es-cont} is continuous in $\theta$.
 \end{proof}
 
\bibliographystyle{amsplain}
\bibliography{thesis}
\vfill

\end{document}